\documentclass[11pt]{amsart}
\usepackage{graphicx}
\usepackage{amsmath,amsthm}
\usepackage{amssymb}
\usepackage{color}
\usepackage{xcolor}
\usepackage[normalem]{ulem}
\usepackage{url}

\newtheorem{theorem}{Theorem}[section]
\newtheorem{proposition}[theorem]{Proposition}
\newtheorem{corollary}[theorem]{Corollary}

\newtheorem{exmp}[theorem]{Example}
\newtheorem{definition}[theorem]{Definition}
\newtheorem{remark}[theorem]{Remark}
\newtheorem{notation}[theorem]{Notation}

\newcommand{\sgn}{\mathrm{sgn}}
\newcommand{\ad}{\mathsf{ad}}
\newcommand{\ot}{\mathsf{ot}}

\begin{document}

\title{Decomposable shuffles}
\author{Jo\~{a}o Dias, Bruno Dinis, Carlos Correia Ramos}

\address[João Dias]{Universidade de \'{E}vora, Centro de Investiga\c{c}\~{a}o em Matem\'{a}tica e Aplica\c{c}\~{o}es (CIMA)
}
\email{joao.miguel.dias@uevora.pt}

\address[Bruno Dinis]{Universidade de \'{E}vora, Escola de Ci\^{e}ncias e Tecnologia, Departamento de Matem\'{a}tica, Centro de Investiga\c{c}\~{a}o em Matem\'{a}tica e Aplica\c{c}\~{o}es (CIMA)}
\email[Corresponding author]{bruno.dinis@uevora.pt}

\address[Carlos Correia Ramos]{Universidade de \'{E}vora, Escola de Ci\^{e}ncias e Tecnologia, Departamento de Matem\'{a}tica, Centro de Investiga\c{c}\~{a}o em Matem\'{a}tica e Aplica\c{c}\~{o}es (CIMA)}
\email{ccr@uevora.pt}

\begin{abstract}
We develop a combinatorial and order-theoretic framework for shuffles,
understood as ordered partitions of the natural numbers. We
introduce decomposable shuffles as generalised sums of finite ordinals,
copies of $\omega$, and copies of its dual $\omega^*$, indexed by
countable scattered linear orders. These building blocks are called benches, ladders and snakes, respectively. We introduce a
subclass, layered shuffles, and for this subclass we construct an explicit
finite-address representation which reduces comparison in the induced
order on $\mathbb N$ to lexicographic comparison of integer tuples.
We determine the order types of represented layered shuffles, establish that all represented orders are scattered, and
define an involution that realises order duality. Finally, we introduce
a hierarchical composition on bench-free injective uniform shuffles
of fixed sign. For each sign, this operation yields a monoid graded by
reduced degree. Its degree-one component is isomorphic to
$\operatorname{Sym}(\mathbb N)$ and consists precisely of the
invertible elements.
\end{abstract}

\maketitle

\section{Introduction}

The usual order on the natural numbers $\mathbb{N}$ is only one of many
possible total orders. These range from simple rearrangements to orders
of arbitrary countable ordinal type. The present
paper introduces a systematic framework to organise and study this landscape, with three
related objectives:  explicit representation, order-theoretic analysis, and
algebraic structure.

Sharkovski\u{\i} considered \cite{Shar} the following total order on $\mathbb{N}$, arising from the coexistence of periodic orbits for continuous maps of the interval
\begin{equation}\label{E:Shar_order}
\begin{aligned}
3 &\prec 5 \prec 7 \prec 9 \prec \cdots \prec (2n+1) \cdot 2^0 \prec \cdots \\
3 \cdot 2 &\prec 5 \cdot 2 \prec 7 \cdot 2 \prec 9 \cdot 2 \prec \cdots
  \prec (2n+1) \cdot 2^1 \prec \cdots \\
3 \cdot 2^2 &\prec 5 \cdot 2^2 \prec 7 \cdot 2^2 \prec 9 \cdot 2^2
  \prec \cdots \prec (2n+1) \cdot 2^2 \prec \cdots \\
&\vdots \\
&\cdots \prec 2^k \prec \cdots \prec 2^3 \prec 2^2 \prec 2 \prec 1.
\end{aligned}
\end{equation}
The order lists the numbers $(2n+1)2^k$, with $k\geq 0$ and $n\geq 1$, first increasing $n$, for fixed $k$, and then increasing $k$, followed by the powers of $2$ in decreasing order.
The Sharkovski\u{\i} order is a distinguished example of a structured family of
total orders on $\mathbb N$ admitting explicit finite-address
representations. To interpret this order as a shuffle (the Sharkovski\u{\i} shuffle), we place
$0$ after $1$, so that its final part is
$\cdots\prec 2^2\prec 2\prec 1\prec 0$.
This extension has the same order type as the original order. In the language developed below, it is a layered
decomposable shuffle of type $(\omega\text{-ladder},1\text{-snake})$:
a doubly indexed family of order type $\omega^2$, followed by a
sequence of order type $\omega^*$, see Example~\ref{Ex:pairs}(d)
and Section~\ref{S:Involution}. Beyond its dynamical interpretation, the Sharkovski\u{\i} order
exhibits a rich internal structure which suggests a broader
combinatorial and order-theoretic framework. 

Informally, a
\emph{shuffle} is a total order on $\mathbb N$ obtained by partitioning
$\mathbb N$ into an ordered family of subsequences. We introduce the
class of \emph{decomposable shuffles}, whose order types are countable scattered generalised sums of blocks of the form $n+\tau+m$, where $n,m<\omega$ and
$\tau\in\{\omega,\omega^*\}$. We
then focus on the more tractable subclass of \emph{layered decomposable
shuffles}. This subclass includes natural examples, while remaining sufficiently structured
to admit explicit representations, direct order-type calculations,
and a natural algebraic operation. In this paper, we develop a finite-address representation theory for
layered decomposable shuffles. The extension to arbitrary countable scattered outer index orders is left for future work.

Countable linear orders have been studied through different approaches
and purposes. Barbina and
Chicot~\cite{BarbinaChicot} classify countable lower $1$-transitive
linear orders via coding trees, which encode how an order is assembled
from simpler pieces through concatenations and lexicographic products.
Delhomm\'{e}~\cite{Delhomme} characterises orthogonality between
countable linear orders, establishing in particular that the class of
countably infinite order types splits into exactly seven
indistinguishability classes for the orthogonality relation, with the
type $\eta$ of the rationals playing a boundary role. Laflamme, Pouzet,
Sauer and Zaguia~\cite{LaflammeEtAl}  study pairs of
orthogonal countable ordinals, working directly with ordinal types
such as $\omega$, $\omega^2$, and higher powers. Brown and Suárez~\cite{BrownSuarez} study finite-condensation
quotients of linear orders, countable scattered linear orders and an algebraic operation obtained by
condensing lexicographic products. In our setting, finite-condensation
classes provide the starting point for the canonical partitions of
shuffles, while finite addresses determine the induced orders on
$\mathbb N$ and support a hierarchical composition operation.

The paper is organised as follows. Section~2 introduces shuffles, decomposable suffles and
their elementary building blocks: ladders, snakes, and benches,
together with canonical partitions determined by finite condensation.
Section~3 develops the finite-address representation through uniform
and mixed elements, defines the induced orders, determines their order
types, and proves that all represented orders are scattered. Section~4 introduces an
involution that reverses the induced order. Section~5 defines the
composition operation, establishes its graded degree arithmetic, and
identifies its group of units. Section~6 presents conclusions and
directions for further research.

\section{Snakes, ladders and benches}\label{S:Snakes}

Infinite permutation groups have been studied extensively in the
literature (see, e.g., Cameron~\cite{Cameron76,Cameron} and
Neumann~\cite{neumann2023infinite}). Inspired by the Sharkovskiĭ
order~\eqref{E:Shar_order}, we adopt a different and more general
approach, considering a broader setting: total orders on $\mathbb{N}$ equipped
with a partition into convex blocks. We call such a structure a
\emph{shuffle}.
 
The classification of a total order on $\mathbb{N}$ by the local
structure of its parts, whether they form infinite increasing
sequences, infinite decreasing sequences, or finite chains, 
 provides a decomposition similar to the prime
factorisation of integers.

The three building blocks introduced in this section, ladders, snakes,
and benches, correspond to the order types $\omega$, $\omega^*$, and
the finite ordinals, respectively. We associate a positive sign with
type $\omega$ and a negative sign with type $\omega^*$. Theorem~\ref{canonical-prop} shows
that the canonical partition of a shuffle into these
blocks is unique unless a consecutive snake--ladder pair occurs.
Moreover, as shown in Section~\ref{S:Composition}, ladders and snakes
provide the basic components for a composition operation on the
bench-free uniform subclass.

\begin{definition}
A \emph{shuffle} is a pair $(\mathcal{P},\prec )$ such that:

\begin{itemize}
\item $\mathcal{P}$ is a partition of $\mathbb{N}$.

\item $\prec$ is a strict total order on $\mathbb{N}$ such that for every $x,y\in
X\in \mathcal{P}$ and $z\in X^{\prime }\in \mathcal{P}$ if we have $x\prec z
\prec y$ then $X= X^{\prime }$.
\end{itemize}
\end{definition}

Note that the second condition states that every part of $\mathcal{P}$ is convex
with respect to $\prec$. Consequently, $\prec$ induces a strict total
order on the elements of $\mathcal{P}$.

Every totally ordered set has an associated \emph{order type}. Two
totally ordered sets $A$ and $B$ have the same order type if there
exists an order-preserving bijection between them. Equivalently, they
are \emph{order-isomorphic}. We classify shuffles through the order
types of the parts of $\mathcal{P}$. Recall that the usual order on
$\mathbb{N}$ has order type $\omega$, while its reverse order has type
$\omega^*$. If $(\mathbb{N},\prec)$ is well-ordered, then its order
type is represented by a unique countable ordinal.

\begin{definition}
Let $(\mathcal{P},\prec )$ be a shuffle. An element of $\mathcal{P}$ is called  a \emph{ladder} if it has order type $\omega $ and a \emph{snake} if
it has order type $\omega ^{\ast }$. Elements of $\mathcal{P}$ of order type some finite ordinal $n$, are called \emph{benches}. 
\end{definition}

The intuition is that ladders (eventually) go ``up" towards infinity, snakes
(eventually) go ``down" from infinity, and benches are parts that remain finite, and so  a part where the sequence temporarily ``rests". 

\begin{remark}
Note that, unlike shuffles, ladders, snakes and benches do not need to have all the natural numbers. So, for example, $(5,6,7,8. \dots)$ is a ladder and $(3,7,11)$ is a bench.

    Moreover, any finite permutation of a ladder, respectively a snake, is also a ladder, respectively a snake.
\end{remark}

\begin{exmp}
Let us give some more examples of ladders, snakes and benches:

\begin{itemize}
\item[$(a)$] The usual ordering of the natural numbers $(0,1,2,3,\dots)$ and that of the odd numbers $(1,3,5,7,\dots)$ are ladders. 

\item[$(b)$] The reverse order of the natural numbers $(\dots,4,3,2,1,0)$ and that of the prime numbers greater than, say $5$, $(\dots, 13, 11,7)$, are snakes. 

\item[$(c)$] Consider the shuffle $(2,4,6,8,\dots)(0,1)(\dots, 7,5,3)$, consisting of the non-zero even natural numbers in increasing order, followed by $0,1$, and then by the odd natural numbers greater than $1$ in decreasing order. This shuffle has three distinct parts: the first part is a ladder, the second is a bench of length $2$, and the third is a snake.
\end{itemize}
\end{exmp}

We recall two standard results concerning well-orders. The first gives
the uniqueness of ordinal representatives, while the second shows, in
particular, that ladders and snakes are not order-isomorphic.

\begin{theorem}[{\protect\cite[Thm 2.12]{Jech}}]
Every well-ordered set is isomorphic to a unique ordinal.
\end{theorem}

\begin{theorem}[{\protect\cite[Thm 4.3.2]{Cie}}]\label{T:wellordered}
A linearly ordered set is well-ordered if and only if it does not contain a subset of order type $\omega^*$.
\end{theorem}

A linear order is said to be \emph{scattered} if it contains no
suborder of type $\eta$, where $\eta$ denotes the order type of
$(\mathbb Q,<)$. We shall use the standard fact that a generalised sum of scattered
linear orders indexed by a scattered linear order is again scattered
\cite[Prop 2.17]{Rosenstein}.

\begin{definition}\label{def:decomposable-shuffle}
A shuffle is said to be \emph{decomposable} if its order type can be
written as
\[
\sum_{t\in L}
\bigl(n_t+\tau_t+m_t\bigr),
\]
where $L$ is a countable scattered linear order,
$n_t,m_t<\omega$, and
\[
\tau_t\in\{\omega,\omega^*\}.
\]
\end{definition}


Decomposable shuffles allow for a more refined classification in terms of ladders, snakes and benches. Note that $\omega^2$ is an example of such an order type. Recall from the arithmetic of order types that if $n$ is a natural number, then $n+\omega=\omega$ and 
$\omega^* + n = \omega^*$.

\begin{definition}
Let $\alpha$ be a finite or countable ordinal and $n$ be a natural number. The parts of a decomposable shuffle which have order type $\omega \cdot \alpha$ are called \emph{$\alpha$-ladders%
}, the parts which have order type $\omega^*\cdot\alpha $ are called \emph{$\alpha$-snakes}, and the parts that have finite order type $n$, 
are called \emph{$n$-benches}.
\end{definition}

When we want to make the parts of a shuffle explicit, we use tuple
notation. For example, a $(2\text{-snake},1\text{-ladder})$ has order
type
\[
(\omega\cdot2)^*+\omega
   =\omega^*\cdot2+\omega.
\]

The following examples illustrate how decomposable shuffles may be
described in terms of ladders, snakes, and benches. 
\begin{exmp}\label{Ex:pairs}
\begin{itemize}

\item[$(a)$] The shuffle (1,3,5,7,\dots)(0,2,4,6,\dots), consisting of the odd numbers followed by the even numbers, both in
increasing order, is a $2$-ladder.

\item[$(b)$] The shuffle (0,2,4,6,\dots)(\dots,5,3,1) is a $(1\text{-ladder},1\text{-snake})$.

\item[$(c)$] The shuffle (6,7,8,9,\dots)(5,4,3,2,1,0) is a $(1\text{-ladder},6\text{-bench})$.

\item[$(d)$] The Sharkovski\u{\i} shuffle, obtained by placing $0$ after $1$ in \eqref{E:Shar_order}, is an
$(\omega\text{-ladder},1\text{-snake})$.

\end{itemize}
\end{exmp}

This description does not distinguish finite permutations, or even
certain infinite ones. For example, (1,0,3,2,5,4,\dots)
still has order type $\omega$ and is therefore a $1$-ladder. The next
section introduces finer representations that retain more information
about the shuffle.

Another related problem is how to identify shuffles that are different, but not in an essential way. To deal with this problem we introduce a notion of partition-equivalence for shuffles.

\begin{definition}\label{def:partition-equivalent}
Two shuffles $(\mathcal{P},\prec )$,  $(\widetilde{\mathcal{P}}, \widetilde{\prec} )$, are said to be \emph{partition-equivalent} if $\widetilde{\prec}=\prec$.
\end{definition}

Note that different shuffles may have different partitions associated with the
same order $\prec$. We now introduce an intrinsic partition determined
by the finite-distance structure of the order. This construction is
available for every shuffle, regardless of whether it is
decomposable or not.

\begin{definition}\label{def:finite-condensation}
Let $(\mathcal{P},\prec)$ be a shuffle. For $x,y\in\mathbb{N}$, write
\[
x\sim_{\mathrm{fin}}y
\]
if the interval between $x$ and $y$ with respect to $\prec$ is finite.
The equivalence classes of $\sim_{\mathrm{fin}}$ are called the
\emph{finite-condensation classes} of the shuffle.
\end{definition}

\begin{definition}\label{def:canonical-partition}
A \emph{canonical partition} of a shuffle $(\mathcal{P},\prec)$ is
obtained from its finite-condensation classes as follows:

\begin{itemize}
\item each class of finite order type, or of order type $\omega$ or
$\omega^*$, is taken as a single part;

\item each class of order type $\omega^*+\omega$ is divided into a snake followed by a ladder.
\end{itemize}
\end{definition}

Equivalently, a canonical partition contains no consecutive
bench--bench, bench--ladder, or snake--bench pair.

\begin{exmp}
Let $\mathcal{P=\{}X_{1},X_{2},X_{3}\mathcal{\}}$ with $X_{1}=\left(
1,3,5\right)$, $X_{2}=(7,9,11,...)$, $X_{3}=(...,8,6,4,2,0)$, and let $\mathcal{P%
}^{\prime }\mathcal{=\{}Y_{1},Y_{2}\mathcal{\}}$ with $%
Y_{1}=(1,3,5,7,9,11,...)$, $Y_{2}=(...,8,6,4,2,0)$. Then the shuffles $(\mathcal{P}%
,\prec )$ and $(\mathcal{P}^{\prime },\prec )$ are partition-equivalent, and $\mathcal{%
P}^{\prime }$ is a canonical partition, while $\mathcal{P}$ is not. 
\end{exmp}

\begin{theorem}\label{canonical-prop}
Every shuffle admits a canonical partition. A canonical partition is
unique if it contains no consecutive snake--ladder pair. If it contains
such a pair, then there are infinitely many distinct canonical
partitions.
\end{theorem}
\begin{proof}
The relation $\sim_{\mathrm{fin}}$ is an equivalence relation, and its
classes are convex. Let $C$ be one such class. If $C$ has both a least and a greatest element, then $C$ is finite.
If it has a least but no greatest element, every element lies at a
finite distance from the least element, and hence $C$ has order type
$\omega$. Dually, if it has a greatest but no least element, it has
order type $\omega^*$. If it has neither a least nor a greatest
element, choosing any $c\in C$ and enumerating its successive
predecessors and successors shows that $C$ has order type $\omega^*+\omega$.

Thus every finite-condensation class is a bench, a ladder, a snake, or
a consecutive snake--ladder pair. The first three types determine a
unique part. A class of type $\omega^*+\omega$ can be divided into a snake and a
ladder at any cut between consecutive elements, and each such choice
gives a canonical partition. Hence the partition is unique precisely
when no class of type $\omega^*+\omega$ occurs.
\end{proof}

\begin{exmp}
Theorem~\ref{canonical-prop} shows that a shuffle with a finite-condensation class of order type $\omega^*+\omega$ admits infinitely many canonical partitions (pair snake--ladder). For example, let $\mathcal{P=\{}Z_{1},Z_{2},Z_{3}\mathcal{\}}$ with $Z_{1}=\left(
1,3,5\right) $, $Z_{2}=(...,8,6,4,2,0)$, $Z_{3}=(7,9,11,...)$, and let $%
\mathcal{P}^{\prime }\mathcal{=\{}Z_{1}^{\prime },Z_{2}^{\prime },Z_{3}^{\prime }%
\mathcal{\}}$  with $Z_{1}^{\prime }=\left( 1,3,5\right) $, $Z_{2}^{\prime
}=(...,2,0,7),$ $Z_{3}^{\prime }=(9,11,...)$. Then the shuffles $(%
\mathcal{P},\prec )$ and $(\mathcal{P}^{\prime },\prec )$ are partition-equivalent,
and both  $\mathcal{P}$ and $\mathcal{P}^{\prime }$ are canonical partitions.
\end{exmp}

\section{Representing shuffles}\label{S:Representing}

This section introduces an explicit finite-address representation for
layered decomposable shuffles (Definition~\ref{def:layered-shuffle}). Each natural number is assigned a finite
tuple of integer indices, and the induced order is recovered by
lexicographic comparison of these addresses. 

We first introduce uniform elements and then construct mixed elements
by concatenating uniform components, as building blocks. The same address formalism will
also be used for the composition operation in
Section~\ref{S:Composition}.

\subsection{Uniform sets}\label{S:Uniform}

Uniform elements of positive degree are finite-depth sequences of natural
numbers whose infinite indexing levels all have the same order type:
either $\omega$ or $\omega^*$. The degree counts these levels, while
their common order type determines the sign: $+$ for $\omega$ and $-$
for $\omega^*$. Degree-zero uniform elements are individual natural
numbers and require no sign.

We shall make use of the following notation. 
\begin{notation}
We write  $\left[ +\infty \right] :=\mathbb{N}$ and $\left[ -\infty \right] :=-\mathbb{N}$. For $m\in \mathbb{N}$,  we set

\begin{eqnarray*}
&\left[ m\right]& :=\left\{0, 1,2,...,m-1\right\},\\
&\left[ -m\right]& :=\left\{ -m+1,-m+2,...,-1,0\right\}.
\end{eqnarray*}

For $n,m\in \mathbb{N}$, with $n\leq m$, we write

\begin{eqnarray*}
\left[ n,m\right] :=\left\{ n,n+1,\cdots,m\right\}. 
\end{eqnarray*}
\end{notation}

When dealing with sequences representing segments, we adopt the following conventions.

\begin{notation}
Ladders and snakes are represented, respectively, by infinite sequences indexed appropriately

\begin{eqnarray*}
\left( s_{k}\right) _{k\in \left[ +\infty \right] } &=&\left(
s_{0},s_{1},s_{2},...,s_{k},...\right)  \\
\left( s_{k}\right) _{k\in \left[ -\infty \right] } &=&\left(
...,s_{-k},...,s_{-2},s_{-1},s_{0}\right) .
\end{eqnarray*}%

Benches are represented by finite sequences of the form 
\begin{equation*}
\left( s_{k}\right) _{k\in \left[ n,m\right] }=\left(
s_{n},s_{n+1},s_{n+2},...,s_{m}\right).
\end{equation*}
\end{notation}

Since we frequently work with multiple indices, we introduce the following simplifying notation.

\begin{notation}
In the following we consider sequences of sequences of the form
\begin{equation*}
s\left( i_{1},...,i_{k}\right) _{i_{1}\in \left[ m_{1}\right]
,...,i_{k}\in \left[ m_{k}\right] },
\end{equation*}

where each index $i_{j}$ belongs to
the domain $\left[ m_{j}\right]$, for $j=1,...,k $. 

Let $I=\left(
i_{1},...,i_{k}\right) \in \Pi _{k}$, where the index set $\Pi _{k}$ denotes the Cartesian product 
\begin{equation*}
\Pi _{k}:=\left[ m_{1}\right] \times \cdots \times \left[ m_{k} \right] . 
\end{equation*}

With this notation, we may write
\begin{equation*}
s\left( I\right) _{I\in \Pi _{k}}=s\left(
i_{1},...,i_{k}\right) _{i_{1}\in \left[ m_{1}\right] ,...,i_{k}\in %
\left[ m_{k}\right] }
\end{equation*}%
When it is convenient to distinguish the last index, we write it explicitly as
\begin{equation*}
s\left( J,i_{k}\right) _{J\in \Pi _{k-1},i_{k}\in \left[ m_{k}%
\right]. }
\end{equation*}
\end{notation}

\begin{definition}[Uniform sets]\label{def:uniform-sets}
\emph{Uniform sets} are defined as follows. First, set
$\mathcal{U}_0:=\mathbb{N}$.
For $d\geq1$ and $\varepsilon\in\{-,+\}$, let
\[
\Pi_d
:=
[\varepsilon\infty]\times\cdots\times[\varepsilon\infty]
=
[\varepsilon\infty]^d,
\]
and define
\[
\mathcal{U}_d^\varepsilon
:=
\left\{
\bigl(s(I)\bigr)_{I\in\Pi_d}
:
s(I)\in\mathbb{N}
\text{ for every }I\in\Pi_d
\right\}.
\]
We further set, for each $d \geq 1$
\[
\mathcal{U}_d
:=
\mathcal{U}_d^-\cup\mathcal{U}_d^+,
\]
and
\[
\mathcal{U}
:=
\bigcup_{d\geq0}\mathcal{U}_d.
\]

The members of $\mathcal{U}$ are called \emph{uniform elements}.
For $S\in\mathcal{U}$, the unique $d\in\mathbb{N}$ such that
$S\in\mathcal{U}_d$ is called the \emph{degree} of $S$ and is defined by
\[
\delta(S):=d.
\]
For $S\in\mathcal{U}_d^\varepsilon$, with $d\geq1$, the
\emph{sign} of $S$ is defined by
\[
\sgn(S):=\varepsilon.
\]
For $S\in\mathcal{U}_0$, we set $\sgn(S):=0$.
\end{definition}

Since $\mathcal U_0=\mathbb N$, a degree-zero uniform element 
represents a single point. Therefore, a singleton bench is uniform,
whereas a bench with more than one point is represented by a
finite mixed element, see below. 
Moreover,
\begin{equation*}
\mathcal{U}_1^\varepsilon=\left\{ s\left( i_{0}\right) _{i_{0}\in [\varepsilon\infty]}\right\} ,
\end{equation*}%
consists of the
$[\varepsilon\infty]$-indexed sequences of natural numbers, while,\begin{equation*}
\mathcal{U}_2^\varepsilon=\left\{ s\left( i_{0}, i_{1}\right) _{i_{0}\in [\varepsilon\infty],i_{1}\in [\varepsilon\infty]}\right\}.
\end{equation*}%
For $d\geq1$, $\mathcal{U}_{d+1}^\varepsilon$ consists of the
$[\varepsilon\infty]$-indexed sequences of elements of
$\mathcal{U}_d^\varepsilon$. Thus, $\delta(S)$ counts the number of
levels of infinite indexing. In particular, all index domains of a
uniform element are infinite and have the same sign.

Structures involving finite index domains, such as finite sequences of
natural numbers or finite families of infinite sequences, will be
represented as mixed elements in
Section~\ref{S:Mixed_degrees}.

\begin{definition}\label{segment}
Let
\[
S=\bigl(s(I)\bigr)_{I\in\Pi_d}
\in\mathcal{U}_d^\varepsilon,
\qquad d\geq1.
\]
The sequences obtained from $S$ by fixing all indices except the last
one are called the \emph{segments} of $S$. For each
$J\in\Pi_{d-1}$, the corresponding segment is
\[
\bigl(s(J,i_d)\bigr)_{i_d\in[\varepsilon\infty]}.
\]
For $d=1$, we adopt the convention that $\Pi_0$ consists of the empty
tuple, so that the only segment of $S$ is $S$ itself.
\end{definition}

\begin{exmp}\label{ex:uniform}

We illustrate uniform elements and their segments, together with some
examples of non-uniform elements.
\begin{enumerate}
\item[(a)] The sequence of even numbers in increasing order,
\begin{equation*}
\left( 2i_{1}\right) _{i_{1}\in \left[ +\infty \right] }=\left(
0,2,4,6,\ldots \right) \in \mathcal{U}_{1}^{+},
\end{equation*}%
and the sequence of even numbers in decreasing order,
\begin{equation*}
\left( -2i_{1}\right) _{i_{1}\in \left[ -\infty \right] }=\left( \ldots
,6,4,2,0\right) \in \mathcal{U}_{1}^{-},
\end{equation*}%
are uniform elements of degree $1$: a ladder and a snake, respectively.

\item[(b)] The family of the powers of the primes, 
$s\left( i_{1},i_{2}\right) =p_{i_{1}+1}^{\,i_{2}+1}$, where $p_{n}$ 
denotes the $n$-th prime,
\begin{equation*}
\left( s\left( i_{1},i_{2}\right) \right) _{i_{1},i_{2}\in \left[ +\infty %
\right] }=\left( \left( 2,4,8,\ldots \right) ,\left( 3,9,27,\ldots
\right) ,\left( 5,25,\ldots \right) ,\ldots \right) \in \mathcal{U}%
_{2}^{+},
\end{equation*}%
is a uniform element of degree $2$: an infinite sequence of ladders. 
For instance, $s\left( 0,2\right) =2^{3}=8$ and $s\left( 2,0\right) =5$.

\item[(c)] In contrast, the pair of ladders
\begin{equation*}
\left( \left( 0,2,4,6,\ldots \right) ,\left( 1,3,5,7,\ldots \right)
\right) ,
\end{equation*}%
the finite sequence $\left( 2,3,4,5\right) $, and the pair
\begin{equation*}
\left( \left( \ldots ,7,5,3\right) ,\left( 0,2,4,6,8,\ldots \right)
\right) ,
\end{equation*}%
formed by a snake followed by a ladder, are not uniform elements. In the first two cases, the outermost index ranges over a finite domain. In the third, the two components have opposite signs. All three can be represented as mixed elements, as explained in Section~\ref{S:Mixed_degrees}. 
\end{enumerate}
\end{exmp}

\subsection{Mixed elements}\label{S:Mixed_degrees}

Uniform elements represent structures of fixed degree and sign, with
all infinite indexing levels having the same sign. However, as we will see, decomposable shuffles may contain successive components of
different degrees and signs. The Sharkovski\u{\i} order is such an example:
it contains a doubly indexed family of sequences, of order type
$\omega^2$, followed by a decreasing sequence of order type $\omega^*$.

We represent such mixed structures by arranging uniform elements
in a family indexed by a finite ordinal, $\omega$, or $\omega^*$.
We call this indexing level the \emph{outer level}, to distinguish it from
the indexing levels within the uniform components. The components
may have different degrees and signs, although mixed elements of
constant degree or constant sign are also allowed.

For
\[
S=\bigl(s(I)\bigr)_{I\in\Pi_d}
\in\mathcal{U}_d^\varepsilon,
\qquad d\geq1,
\]
we define
\[
\langle S\rangle
:=
\{s(I):I\in\Pi_d\}.
\]
For $n\in\mathcal{U}_0$, we set
$\langle n\rangle:=\{n\}$.
Thus,
\[
\langle\cdot\rangle:
\mathcal{U}\longrightarrow\mathcal{P}(\mathbb{N})
\]
associates to each uniform element the set of natural numbers
occurring in it.

We now consider the set
$\bigcup_{k\geq1}\mathcal{U}^k$
of finite sequences of elements of $\mathcal{U}$, together with the
sets
\[
\mathcal{U}^{[+\infty]}
=
\left\{
(S_t)_{t\in[+\infty]}:
S_t\in\mathcal{U}
\text{ for every }t\in[+\infty]
\right\}
\]
and
\[
\mathcal{U}^{[-\infty]}
=
\left\{
(S_t)_{t\in[-\infty]}:
S_t\in\mathcal{U}
\text{ for every }t\in[-\infty]
\right\}.
\]
We define
\[
\mathcal{U}_\infty
:=
\left(
\bigcup_{k\geq1}\mathcal{U}^k
\right)
\cup
\mathcal{U}^{[+\infty]}
\cup
\mathcal{U}^{[-\infty]}.
\]
The elements of $\mathcal{U}_\infty$ are called
\emph{mixed elements}.

Every $S\in\mathcal{U}_\infty$ can be written in the form
\begin{equation}\label{S}
S=(S_t)_{t\in[m]},
\end{equation}
where
\[
m\in
\bigl(\mathbb{N}\setminus\{0\}\bigr)
\cup\{+\infty,-\infty\},
\]
and $S_t\in\mathcal{U}$ for every $t\in[m]$.

The map $\langle\cdot\rangle$ extends to $\mathcal{U}_\infty$ by
\[
\langle\cdot\rangle:
\mathcal{U}_\infty
\longrightarrow
\mathcal{P}(\mathbb{N}).
\]
\[
\langle S\rangle
:=
\bigcup_{t\in[m]}\langle S_t\rangle,
\qquad
S=(S_t)_{t\in[m]}.
\]

For each $t\in[m]$, set $k_t:=\delta(S_t)$. If $k_t\geq1$, then
\[
S_t
=
\bigl(s(t)(I)\bigr)_{I\in\Pi_{t,k_t}},
\]
where $I=(i_{t,1},\ldots,i_{t,k_t})$ and $\Pi_{t,k_t}=[\varepsilon_t\infty]^{k_t},
\varepsilon_t:=\sgn(S_t)$.

For $k_t=0$, we adopt the convention
\[
\Pi_{t,0}:=\{\varnothing\},
\qquad
s(t)(\varnothing):=S_t.
\]
Consequently, we may write
\begin{equation}\label{SS}
S=
\bigl(s(t)(I)\bigr)_{
t\in[m],\,I\in\Pi_{t,k_t}}.
\end{equation}

\begin{definition}\label{def:mixed-degree-sign}
Let $S=(S_t)_{t\in[m]}\in\mathcal{U}_\infty$.
The \emph{degree} and the \emph{sign} of $S$ are defined, respectively, by
\[
\delta(S)
:=
\left(
m,
\bigl(\delta(S_t)\bigr)_{t\in[m]}
\right)
\]
and
\[
\sgn(S)
:=
\bigl(\sgn(S_t)\bigr)_{t\in[m]}.
\]

\end{definition}
The first entry $m$ of $\delta(S)$ indicates the index set of the
components. When $m=\pm\infty$, its sign is the sign of the outer
index level.  The sequence
$\sgn(S)$ records the signs of the uniform components of $S$. 

An element $S\in\mathcal{U}_\infty$, written as in \eqref{SS}, is said
to be \emph{injective} if
\[
(t,I)\neq(p,J)
\quad\Longrightarrow\quad
s(t)(I)\neq s(p)(J)
\]
for all $t,p\in[m]$, $I\in\Pi_{t,k_t}$, $J\in\Pi_{p,k_p}$.

We define
\begin{equation}\label{S-infinity}
\mathcal{S}_\infty
:=
\left\{
S\in\mathcal{U}_\infty:
\langle S\rangle=\mathbb{N}
\text{ and }S\text{ is injective}
\right\}.
\end{equation}

Both conditions in the definition of $\mathcal{S}_\infty$ are
necessary: a mixed element may fail to cover $\mathbb{N}$, may assign
the same value to distinct pairs of indices, or may have both defects.

Thus, an element $S\in\mathcal{S}_\infty$ is a family of uniform
elements, indexed either by a finite set, by $[+\infty]$, or by
$[-\infty]$, whose components jointly contain every natural number exactly once. As shown in Proposition~\ref{P:partition}, the
deepest-level sequences of the components of $S$, together with its
degree-zero components, determine a partition of $\mathbb{N}$ into
segments and benches.

To recover the position of each natural number within a mixed
representation, we introduce the \emph{address map}. It assigns to each
number its outer component index together with its indices inside that
component. The order associated with the representation will then be
recovered by lexicographic comparison of addresses, as shown in
Theorem~\ref{T:Totalorder}.

\begin{definition}\label{D:address}
Let $S\in\mathcal{S}_\infty$. Since $\langle S\rangle=\mathbb{N}$ and
$S$ is injective, every $n\in\mathbb{N}$ has a unique representation $n=s(t)(I)$,
$t\in[m]$, $I\in\Pi_{t,k_t}$.
The \emph{address map} is
\[
\ad_S:
\mathbb{N}
\longrightarrow
\bigsqcup_{t\in[m]}
\bigl(\{t\}\times\Pi_{t,k_t}\bigr)
\]
defined by
\begin{equation}\label{address}
\ad_S\bigl(s(t)(I)\bigr)
:=
(t,I)
=
(t,i_{t,1},\ldots,i_{t,k_t}).
\end{equation}
For $k_t=0$, the address $(t,\varnothing)$ is identified with $(t)$.
\end{definition}

By construction, the map $\ad_S$ is a bijection. When there is no risk of confusion, we write $\ad$ instead of $\ad_S$. Its
inverse is given by
\[
\ad_S^{-1}(t,I)=s(t)(I).
\]

\begin{exmp}\label{example-ad}

Let $S=(S_0,S_1)\in\mathcal{U}_\infty$,
where
\[
S_0=(2i)_{i\in[+\infty]}
\in\mathcal{U}_1^+,
\qquad
S_1=(2i+1)_{i\in[+\infty]}
\in\mathcal{U}_1^+.
\]
Thus,
\[
S=
\bigl(
(0,2,4,6,\ldots),
(1,3,5,7,\ldots)
\bigr).
\]
Its general term is $s(t)(i)=2i+t$, $t\in[2]$, $i\in[+\infty]$.

Since $\langle S\rangle=\mathbb{N}$ and $S$ is injective, we have
$S\in\mathcal{S}_\infty$. For example,
\[
3=s(1)(1),
\qquad
4=s(0)(2),
\]
and hence
\[
\ad(3)=(1,1),
\qquad
\ad(4)=(0,2).
\]

\end{exmp}

\subsection{Orders induced by elements of
$\mathcal{S}_\infty$}\label{S:Orders_from_Sinfty}

Different elements of $\mathcal{S}_\infty$ may induce the same order on
$\mathbb{N}$, or order-isomorphic orders, for three reasons:

First, different partitions may be associated with the same order.
In particular, changing the boundary in a consecutive snake--ladder
pair does not change the induced order (see
Theorem~\ref{canonical-prop}).

Second, the same indexed family may admit different hierarchical
groupings. For example, a $[+\infty]$-indexed family of elements of
$\mathcal{U}_d^+$ may be represented either as a mixed element of
constant degree $d$ or as a single element of
$\mathcal{U}_{d+1}^+$. Dually, the same holds for a
$[-\infty]$-indexed family of elements of $\mathcal{U}_d^-$.

Third, ordinal absorption may produce distinct representations of
order-isomorphic orders. For example,
$\omega+\omega^2=\omega^2$,
so a ladder followed by a degree-$2$ positive component is
order-isomorphic to a single degree-$2$ positive component.

These examples show that degree and sign are invariants of the chosen
representation, but not, in general, of the induced order. A theory of
canonical representations would require a systematic treatment of
these regrouping and absorption phenomena, which we leave for future work.

For now, we associate a total order on $\mathbb{N}$ with each element
$S\in\mathcal{S}_\infty$. The order is obtained by lexicographic
comparison of addresses.

Given two addresses
\[
(t,i_{t,1},\ldots,i_{t,k_t})
\quad\text{and}\quad
(p,i_{p,1},\ldots,i_{p,k_p}),
\]
we write
\[
(t,i_{t,1},\ldots,i_{t,k_t})
\vartriangleleft
(p,i_{p,1},\ldots,i_{p,k_p})
\]
if either $t<p$, or \(t=p\) and there exists
$r\in\{1,\ldots,k_t\}$ such that
\[
i_{t,j}=i_{p,j}
\quad\text{for every }j<r,
\qquad
i_{t,r}<i_{p,r}.
\]
Thus, the outer indices are compared first and, when they coincide,
the internal indices are compared lexicographically.

\begin{definition}\label{D:orderS}
Let $S\in\mathcal{S}_\infty$ and $x,y\in\mathbb{N}$. We define
\[
x\prec_S y
\quad\Longleftrightarrow\quad
\ad_S(x)\vartriangleleft\ad_S(y).
\]
\end{definition}

\begin{theorem}\label{T:Totalorder}
For every $S\in\mathcal{S}_\infty$, the relation $\prec_S$ is a strict
total order on $\mathbb{N}$.
\end{theorem}

\begin{proof}
The address map $\ad_S$ is a bijection between $\mathbb{N}$ and the set
of addresses determined by $S$. The relation $\vartriangleleft$ is a
strict total order on this set: the outer indices are compared first,
and addresses with the same outer index are compared lexicographically
in their internal coordinates. 

Since $\prec_S$ is the pullback of $\vartriangleleft$ through $\ad_S$,
it is a strict total order on $\mathbb{N}$.
\end{proof}

\begin{exmp}
\begin{enumerate}

\item[(a)] Consider again Example~\ref{example-ad}. We have
\[
\ad(5)=(1,2),
\qquad
\ad(8)=(0,4),
\]
and
\[
\ad(21)=(1,10),
\qquad
\ad(22)=(0,11).
\]
Since
\[
(0,4)
\vartriangleleft
(0,11)
\vartriangleleft
(1,2)
\vartriangleleft
(1,10),
\]
we obtain
\[
8\prec_S22\prec_S5\prec_S21.
\]

\item[(b)] Consider the mixed element $S=(S_0,S_1)\in\mathcal{U}_\infty$, where
\[
S_0=(1,0,2,4,8,\ldots)\in\mathcal{U}_1^+
\]
and
\[
S_1=
\bigl((-2j+3)2^{-i}\bigr)_
{i,j\in[-\infty]}
\in\mathcal{U}_2^-.
\]
The first component contains $0$ and all powers of $2$, while the
second contains all natural numbers of the form
\[
(2r+1)2^q,
\qquad r\geq1,\quad q\geq0.
\]
Thus, every natural number occurs exactly once, and hence $S\in\mathcal{S}_\infty$.

The components have degrees  $k_0=1, k_1=2$,
and opposite signs. This representation induces a variation of the Sharkovski\u{\i} shuffle introduced above, whose least element is $1$ and whose greatest
element is $3$.

For example,
\[
\ad(4)=(0,3),
\qquad
\ad(14)=(1,-1,-2),
\]
and
\[
\ad(15)=(1,0,-6),
\qquad
\ad(3)=(1,0,0).
\]
Therefore,
\[
(0,3)
\vartriangleleft
(1,-1,-2)
\vartriangleleft
(1,0,-6)
\vartriangleleft
(1,0,0),
\]
and hence
\[
4\prec_S14\prec_S15\prec_S3.
\]

\end{enumerate}
\end{exmp}

The following result associates a shuffle with each element of
$\mathcal{S}_\infty$.

\begin{proposition}\label{P:partition}
Let $S=
\bigl(s(t)(I)\bigr)_{
t\in[m],\,I\in\Pi_{t,k_t}}
\in\mathcal{S}_\infty$. For $t\in[m]$ with $k_t\geq1$ and
$J\in\Pi_{t,k_t-1}$, define $P(t,J):= \left\{
s(t)(J,i): i\in[\varepsilon_t\infty]
\right\},\varepsilon_t:=\sgn(S_t)$.
We also set $P(t):=\{S_t\}$
whenever $k_t=0$. Then 
$$\mathcal{P}_S
:=
\left\{
P(t,J):
t\in[m],\ k_t\geq1,\ 
J\in\Pi_{t,k_t-1}
\right\}
\cup
\left\{
P(t):
t\in[m],\ k_t=0
\right\}$$
is a partition of $\mathbb{N}$, and
$(\mathcal{P}_S,\prec_S)$
is a shuffle.
\end{proposition}

\begin{proof}
Since $S\in\mathcal{S}_\infty$, every natural number has a unique
address. Hence every natural number belongs to exactly one member of
$\mathcal{P}_S$, so $\mathcal{P}_S$ is a partition of $\mathbb{N}$.

It remains to show that every part is convex with respect to
$\prec_S$. A part $P(t,J)$ consists precisely of the elements whose
addresses have the fixed prefix $(t,J)$
and whose last coordinate varies over
$[\varepsilon_t\infty]$. In the lexicographic order, the set of all
addresses with a fixed prefix is convex. Therefore, if
\[
x,y\in P(t,J)
\qquad\text{and}\qquad
x\prec_S z\prec_S y,
\]
then the address of $z$ has the same prefix $(t,J)$, and hence
$z\in P(t,J)$. The sets $P(t)$ corresponding to degree-zero
components are singletons and are therefore convex.

Thus every part of $\mathcal{P}_S$ is convex with respect to
$\prec_S$, and consequently
$(\mathcal{P}_S,\prec_S)$ is a shuffle.
\end{proof}

The partition $\mathcal{P}_S$ records the
last indexing level of each positive-degree component and treats each
degree-zero component as a singleton part. A bench can be represented by a single uniform element precisely
when it is a singleton, in which case the element has degree zero.
Larger benches are represented by mixed elements consisting of
finite families of degree-zero components. Thus a finite bench appears in $\mathcal{P}_S$ as a
consecutive family of singleton parts. If there are isolated benches it is not in canonical form. However there is a unique partition-equivalent canonical partition obtained from $\mathcal{P}_S$ condensing the isolated benches.

For a uniform or mixed element $S$, we denote by $\ot(S)$ the order
type determined by the lexicographic order on its indices. If
$S\in\mathcal{S}_\infty$, then
\[
\ot(S)=\ot(\mathbb{N},\prec_S).
\]

\begin{theorem}\label{order-type}
Let $S\in\mathcal{U}$ be a uniform element. Then
\[
\ot(S)=
\begin{cases}
\omega^d,
    & S\in\mathcal{U}_d^+,\quad d\geq1,\\[2mm]
(\omega^d)^*,
    & S\in\mathcal{U}_d^-,\quad d\geq1,\\[2mm]
1,
    & S\in\mathcal{U}_0.
\end{cases}
\]

For a mixed element
\[
S=(S_t)_{t\in[m]}\in\mathcal{U}_\infty,
\]
where
\[
m\in
\bigl(\mathbb{N}\setminus\{0\}\bigr)
\cup\{+\infty,-\infty\},
\]
we have
\[
\ot(S)
=
\sum_{t\in[m]}\ot(S_t),
\]
where the sum is taken in the order of the index set $[m]$.
\end{theorem}

\begin{proof}
We argue by induction on the degree.

An element of $\mathcal{U}_0$ has a singleton index set and therefore
has order type $1$. An element of $\mathcal{U}_1^+$ is indexed by
$[+\infty]$ and hence has order type $\omega$, while an element of
$\mathcal{U}_1^-$ is indexed by $[-\infty]$ and has order type
$\omega^*$.

Suppose that the result holds in degree $d\geq1$. Let
$S\in\mathcal{U}_{d+1}^+$. For each fixed value of the first index,
the remaining indices determine an element of
$\mathcal{U}_d^+$, of order type $\omega^d$. Since the first index
ranges over $[+\infty]$, it follows that
\[
\ot(S)
=
\sum_{i_1\in[+\infty]}\omega^d
=
\omega^d\cdot\omega
=
\omega^{d+1}.
\]

The negative case is the order dual of the positive one. Thus, if
$S\in\mathcal{U}_{d+1}^-$, then
\[
\ot(S)
=
(\omega^{d+1})^*.
\]

Finally, a mixed element is obtained by concatenating its components
in the order of the outer index $t\in[m]$. Therefore,
\[
\ot(S)
=
\sum_{t\in[m]}\ot(S_t).
\]
\end{proof}

\begin{definition}\label{def:layered-shuffle}
    A decomposable shuffle is said to be \emph{layered} if its order type
admits a representation of the form
\[
\sum_{t\in M}\alpha_t,
\]
where $M$ is finite or has order type $\omega$ or $\omega^*$, and
\[
\alpha_t\in
\{n:n<\omega\}
\cup
\{\omega^d,(\omega^d)^*:d\geq1\}.
\]
\end{definition}

\begin{proposition}\label{P:represented-orders}
Let $\prec$ be a strict total order on $\mathbb N$. Then $\prec$ is
induced by an element of $\mathcal S_\infty$ if and only if $$\operatorname{ot}(\mathbb N,\prec)
=
\sum_{t\in M}\alpha_t,$$
where $M$ is finite or has order type $\omega$ or $\omega^*$, and
each $\alpha_t$ is of one of the following forms: 
\begin{itemize}
    \item $0<n<\omega$,
    \item $\omega^d$,
    \item $(\omega^d)^*$,
for some $d\geq1$.
\end{itemize}
 In particular, the total orders induced by elements of
$\mathcal S_\infty$ correspond precisely to the layered decomposable shuffles of
Definition~\ref{def:layered-shuffle}.
\end{proposition}

\begin{proof}
Suppose first that $\prec= \prec_S$ for some $S=(S_t)_{t\in[m]}\in\mathcal S_\infty$.

By Theorem~\ref{order-type} we have that
$\operatorname{ot}(\mathbb N,\prec)
=
\sum_{t\in[m]}\operatorname{ot}(S_t)$,
where each summand has order type $1$,$\omega^d$, or, $(\omega^d)^*$
for some $d\geq1$. Hence the required representation exists.

Conversely, suppose that $\operatorname{ot}(\mathbb N,\prec)
=
\sum_{t\in M}\alpha_t$ has the stated form. Replace each finite summand $n$ by $n$
consecutive singleton summands. The resulting outer index order is
still finite or has order type $\omega$ or $\omega^*$.

Let $D$ be the corresponding generalised sum of standard index
domains, using a singleton for a summand of type $1$,
$[+\infty]^d$ for a summand of type $\omega^d$, and
$[-\infty]^d$ for a summand of type $(\omega^d)^*$. Choose an
order isomorphism
\[
h:D\longrightarrow(\mathbb N,\prec).
\]
Applying $h$ componentwise defines a mixed element
$S\in\mathcal U_\infty$. Since $h$ is bijective, $S$ is injective
and has support $\mathbb N$. Thus $S\in\mathcal S_\infty$, and its
induced order is precisely $\prec$.
\end{proof}

Not every total order on $\mathbb{N}$ is representable by an element of
$\mathcal{S}_\infty$. The following result gives a basic
order-theoretic obstruction.

\begin{theorem}\label{T:scattered}
Every total order on $\mathbb N$ represented by an element of
$\mathcal S_\infty$ is scattered.
\end{theorem}

\begin{proof}
Let $S=(S_t)_{t\in[m]}\in\mathcal S_\infty$.
By Theorem~\ref{order-type},
\[
\operatorname{ot}(S)
=
\sum_{t\in[m]}\operatorname{ot}(S_t),
\]
where every summand has order type $1, \omega^d$ or $(\omega^d)^*$ for some $d\geq1$. Each of these order types is scattered, and the
outer index order $[m]$ is finite or has order type $\omega$ or
$\omega^*$. Since a generalised sum of scattered orders indexed by a
scattered order is scattered, the result follows.
\end{proof}

\begin{corollary}\label{C:eta}
No total order on $\mathbb{N}$ of order type $\eta$ (the order type of the rationals) is induced by an
element of $\mathcal{S}_\infty$.
\end{corollary}
\begin{proof}
By Theorem~\ref{T:scattered}, every order induced by an element of
$\mathcal{S}_\infty$ is scattered, whereas $\eta$ is not scattered.
\end{proof}

\begin{exmp}\label{E:examples}
The following countable ordinals distinguish the classes of layered,
decomposable, and scattered order types. Since they are countable
ordinals, they are all scattered.

\medskip
\noindent
\textup{(a)} The ordinal $\omega^\omega$ is layered. In fact,
\[
\omega^\omega
=
\sum_{d<\omega}\omega^{d+1}.
\]
This is a layered decomposition indexed by $\omega$, with components
of degrees $1,2,3,\ldots.$
It is also decomposable, since
\[
\omega^\omega
=
\sum_{t<\omega^\omega}\omega
=
\omega\cdot\omega^\omega.
\]
By Proposition~\ref{P:represented-orders}, it is induced by an
element of $\mathcal S_\infty$.

\medskip
\noindent
\textup{(b)} The ordinal $\omega^{\omega+1}$ is decomposable but not
layered. To see that it is decomposable, let $L$ have order type
$\omega^{\omega+1}$ and, for every $t\in L$, set $n_t=m_t=0
\qquad\text{and}\qquad
\tau_t=\omega.$
Then
\[
\sum_{t\in L}(n_t+\tau_t+m_t)
=
\sum_{t\in L}\omega
=
\omega\cdot\omega^{\omega+1}
=
\omega^{\omega+1}.
\]

Suppose that $\omega^{\omega+1}$ were layered. Since it is
well-ordered no component could have type $(\omega^d)^*$ and the
outer index order could not have type $\omega^*$. A finite outer sum
of finite ordinals and powers $\omega^d$, with $d<\omega$, is
strictly below $\omega^\omega$. An outer sum of type $\omega$ is at
most $\omega^\omega$, since each of its finite partial sums is below
$\omega^\omega$. Therefore, no layered representation can have order type
$\omega^{\omega+1}$. Consequently, this order is not induced by any
element of $\mathcal S_\infty$.

\medskip
\noindent
\textup{(c)} The ordinal $\omega^\omega+1$ is scattered but not
decomposable. Suppose, towards a contradiction, that
\[
\omega^\omega+1
=
\sum_{t\in L}(n_t+\tau_t+m_t),
\]
where $L$ is countable and scattered,
$n_t,m_t<\omega$, and $\tau_t\in\{\omega,\omega^*\}$.

Since $\omega^\omega+1$ is well-ordered, no summand can contain a
copy of $\omega^*$. Hence $\tau_t=\omega$ for every $t\in L$.
Moreover, choosing one element from each summand embeds $L$ into the
resulting well-order, so $L$ must itself be well-ordered.

Since $\omega^\omega+1$ has a greatest element, $L$ must have a
greatest element $t_0$, and $m_{t_0}>0$. Writing $\beta
=
\sum_{t<t_0}(n_t+\omega+m_t)$,
we obtain $\omega^\omega+1 =
\beta+\omega+m_{t_0}$.

If $\beta<\omega^\omega$, then, for some $1\leq D<\omega$,
$\beta<\omega^D$
and therefore $\beta+\omega+m_{t_0}
<
\omega^{D+1}
<
\omega^\omega$.

If $\beta\geq\omega^\omega$, then $\beta+\omega+m_{t_0}
\geq
\omega^\omega+\omega
>
\omega^\omega+1$.

Both alternatives are impossible. Thus $\omega^\omega+1$ is not
decomposable.

\end{exmp}

\section{An involution}\label{S:Involution}

Every total order has a natural \emph{dual}, obtained by reversing
its ordering. In the present framework, the dual of an order induced
by an element of $\mathcal S_\infty$ is again induced by an element
of $\mathcal S_\infty$. At the level of order types, this duality
exchanges $\omega$ and $\omega^*$ and, more generally, exchanges
$\omega^d$ and $(\omega^d)^*$.

In this section, we make this duality explicit by introducing an
involution on uniform elements and extending it to mixed elements.
The Sharkovski\u{\i} order and its reverse are related by this
involution, see Example~\ref{Ex:Shar_involution}. The construction
will also reappear in Section~\ref{S:Composition}, where it provides
an isomorphism between the groups
$(\mathcal I_1^+,\circ)$ and $(\mathcal I_1^-,\circ)$.

\begin{definition}\label{D:involution}
The map $(\cdot)^*:\mathcal U\longrightarrow\mathcal U$
is defined as follows. For
\[
S=\bigl(s(I)\bigr)_{I\in\Pi_d}
\in\mathcal U_d^\varepsilon,
\qquad d\geq1,
\]
we set
\[
S^*
:=
\bigl(
s(-i_1,\ldots,-i_d)
\bigr)_{
i_1,\ldots,i_d\in[-\varepsilon\infty]
}.
\]
For $S\in\mathcal U_0$, we set $S^*:=S$.
\end{definition}

\begin{proposition}\label{P:involution}
The map $(\cdot)^*:\mathcal U\longrightarrow\mathcal U$
is an involution. Moreover,
$\bigl(\mathcal U_d^\varepsilon\bigr)^*
=
\mathcal U_d^{-\varepsilon}$, for every $d\geq1$ and $\varepsilon\in\{-,+\}$.
\end{proposition}

\begin{proof}
If $S\in\mathcal U_d^\varepsilon$, changing the sign of each index
changes every index domain from $[\varepsilon\infty]$ to
$[-\varepsilon\infty]$. Hence $S^*\in\mathcal U_d^{-\varepsilon}$,
so the degree is preserved and the sign is reversed.

Applying the same construction a second time changes each index
$i_j$ into $-(-i_j)=i_j$. Therefore
$(S^*)^*=S$, and $(\cdot)^*$ is an involution.
\end{proof}

We now reverse the outer index domain of a mixed element. For
\[
m\in
(\mathbb N\setminus\{0\})\cup\{+\infty,-\infty\},
\]
define
\[
m^*
:=
\begin{cases}
m,
    & m\in\mathbb N\setminus\{0\},\\
-\infty,
    & m=+\infty,\\
+\infty,
    & m=-\infty.
\end{cases}
\]
Let
\[
\rho_m:[m^*]\longrightarrow[m]
\]
be given by
\[
\rho_m(t)
:=
\begin{cases}
m-1-t,
    & m\in\mathbb N\setminus\{0\},\\
-t,
    & m\in\{+\infty,-\infty\}.
\end{cases}
\]
The map $\rho_m$ is an order-reversing bijection, and $\rho_m^{-1}=\rho_{m^*}$.

\begin{definition}\label{D:mixed-involution}
Let  $S=(S_t)_{t\in[m]}\in\mathcal U_\infty$.
The \emph{involution} of $S$ is the mixed element
$S^*:=\bigl(S_{\rho_m(t)}^*\bigr)_{t\in[m^*]}$.

\end{definition}

Thus, the involution reverses both the order within each component
and the order of the components themselves. The reversal of the
outer index domain is essential: reversing only the components would
not, in general, reverse the order induced by the mixed element.

\begin{exmp}\label{Ex:inv-3ladder}
Let
\[
S=
\bigl(
(0,3,6,\ldots),
(1,4,7,\ldots),
(2,5,8,\ldots)
\bigr)
\in\mathcal S_\infty.
\]
The three components belong to $\mathcal U_1^+$. The involution
reverses their outer order and changes each ladder into a snake:
\[
S^*
=
\bigl(
(\ldots,8,5,2),
(\ldots,7,4,1),
(\ldots,6,3,0)
\bigr).
\]
\end{exmp}

\begin{theorem}\label{T:inv-order}
The map $(\cdot)^*:\mathcal U_\infty\longrightarrow\mathcal U_\infty$
is an involution and restricts to an involution
\[
(\cdot)^*:\mathcal S_\infty\longrightarrow\mathcal S_\infty.
\]
Moreover, for every $S\in\mathcal S_\infty$ and all
$x,y\in\mathbb N$,
\[
x\prec_S y
\quad\Longleftrightarrow\quad
y\prec_{S^*}x.
\]
\end{theorem}

\begin{proof}
Let
\[
S=(S_t)_{t\in[m]}\in\mathcal U_\infty.
\]
By Proposition~\ref{P:involution}, the componentwise involution
preserves the support of every component. Reversing the outer index
domain also leaves the total support unchanged. Hence $\langle S^*\rangle=\langle S\rangle$.

The construction merely reindexes the entries of $S$, so it also
preserves injectivity. Consequently,
\[
S\in\mathcal S_\infty
\quad\Longrightarrow\quad
S^*\in\mathcal S_\infty.
\]

Since
\[
(m^*)^*=m,
\qquad
\rho_{m^*}=\rho_m^{-1},
\]
and the componentwise operation is an involution, applying the mixed
construction twice gives $(S^*)^*=S$.

Now let $x\in\mathbb N$ have address
\[
\ad_S(x)=(t,i_1,\ldots,i_k).
\]
In $S^*$, the component containing $x$ has outer index
$\rho_{m^*}(t)$, and all the internal indices have their signs
reversed. Therefore,
\[
\ad_{S^*}(x)
=
\bigl(
\rho_{m^*}(t),
-i_1,\ldots,-i_k
\bigr).
\]

Let
\[
\ad_S(y)=(p,j_1,\ldots,j_\ell).
\]
If $t\neq p$, then the order of the outer coordinates is reversed by
$\rho_{m^*}$. If $t=p$, then $k=\ell$, and the first internal
coordinate at which the two addresses differ has its inequality
reversed by the map $i\mapsto-i$. Thus,
\[
\ad_S(x)\vartriangleleft\ad_S(y)
\quad\Longleftrightarrow\quad
\ad_{S^*}(y)\vartriangleleft\ad_{S^*}(x).
\]
By Definition~\ref{D:orderS}, this is equivalent to
\[
x\prec_S y
\quad\Longleftrightarrow\quad
y\prec_{S^*}x.
\]
\end{proof}

\begin{corollary}\label{C:inv-ordertype}
For every $S\in\mathcal S_\infty$, the order induced by $S^*$ is the
order dual of the order induced by $S$. In particular,
\[
\operatorname{ot}(S^*)
=
\operatorname{ot}(S)^*.
\]
\end{corollary}

\begin{exmp}\label{Ex:Shar_involution}
Let $S\in\mathcal S_\infty$ be a representation of the
Sharkovski\u{\i} shuffle. Its order type is $\omega^2+\omega^*$. By
Theorem~\ref{T:inv-order}, $S^*$ represents the reverse order,
of type
\[
(\omega^2+\omega^*)^*
=
\omega+(\omega^2)^*.
\]
\end{exmp}

\section{On the composition of shuffles}
\label{S:Composition}

The finite-address formalism developed in Section~3 suggests a
hierarchical extension of the ordinary permutation composition. We
restrict the construction to uniform elements of positive
degree and fixed sign which are injective and have support $\mathbb N$, so that the final coordinate of one element can
be reindexed by the coordinate structure of another.

For $d\geq1$ and $\varepsilon\in\{+,-\}$, let
\[
\mathcal I_d^\varepsilon
:=
\left\{
S\in\mathcal U_d^\varepsilon:
\langle S\rangle=\mathbb N
\text{ and }S\text{ is injective}
\right\},
\]
and set
\[
\mathcal I^\varepsilon
:=
\bigcup_{d\geq1}\mathcal I_d^\varepsilon.
\]

The elements of $\mathcal I^\varepsilon$ are bench-free uniform
shuffles of sign $\varepsilon$.

\begin{definition}\label{def:composition}
Let
\[
S=
\bigl(
s(i_1,\ldots,i_p)
\bigr)_{
i_1,\ldots,i_p\in[\varepsilon\infty]
}
\in\mathcal I_p^\varepsilon
\]
and
\[
U=
\bigl(
u(j_1,\ldots,j_q)
\bigr)_{
j_1,\ldots,j_q\in[\varepsilon\infty]
}
\in\mathcal I_q^\varepsilon
\]
be elements of the same sign
$\varepsilon\in\{+,-\}$.

The \emph{composition} of $S$ and $U$ is the element
\[
S\circ U
=
R
=
\bigl(
r(i_1,\ldots,i_{p-1},j_1,\ldots,j_q)
\bigr)_{
i_1,\ldots,i_{p-1},j_1,\ldots,j_q
\in[\varepsilon\infty]
}
\]
defined by
\[
r(i_1,\ldots,i_{p-1},j_1,\ldots,j_q)
:=
s\bigl(
i_1,\ldots,i_{p-1},
\varepsilon u(j_1,\ldots,j_q)
\bigr).
\]
\end{definition}

Thus, the composition replaces the last indexing level of $S$ by the
$q$ indexing levels of $U$. Its resulting index pattern has
$p-1+q$ coordinates, all of sign $\varepsilon$.

\begin{proposition}\label{P:composition-monoid}
Let $\varepsilon\in\{+,-\}$.

\begin{enumerate}
\item If $S\in\mathcal I_p^\varepsilon$ and $U\in\mathcal I_q^\varepsilon$,
then $S\circ U\in\mathcal I_{p-1+q}^\varepsilon$.

\item The composition is associative on $\mathcal I^\varepsilon$.

\item The element
\[
E^\varepsilon
:=
\bigl(\varepsilon i\bigr)_{i\in[\varepsilon\infty]}
\in\mathcal I_1^\varepsilon
\]
is an identity for the composition.
\end{enumerate}

Consequently,
\[
(\mathcal I^\varepsilon,\circ)
\]
is a monoid. Moreover, the reduced degree
\[
\ell(S):=\delta(S)-1
\]
satisfies
\[
\ell(S\circ U)=\ell(S)+\ell(U).
\]
Hence
\[
\mathcal I^\varepsilon
=
\bigsqcup_{n\geq0}\mathcal I_{n+1}^\varepsilon
\]
is an $\mathbb N$-graded monoid.
\end{proposition}

\begin{proof}
Let $S\in\mathcal I_p^\varepsilon,
U\in\mathcal I_q^\varepsilon$.
Since $S$ and $U$ belong to $\mathcal I^\varepsilon$, their coordinate maps
$s:[\varepsilon\infty]^p\longrightarrow\mathbb N$ and $u:[\varepsilon\infty]^q\longrightarrow\mathbb N$
are bijections.

The map
\[
\vartheta_U:
[\varepsilon\infty]^{p-1}
\times
[\varepsilon\infty]^q
\longrightarrow
[\varepsilon\infty]^p
\]
defined by
\[
\vartheta_U
(i_1,\ldots,i_{p-1},j_1,\ldots,j_q)
=
\bigl(
i_1,\ldots,i_{p-1},
\varepsilon u(j_1,\ldots,j_q)
\bigr)
\]
is a bijection. Indeed, $u$ is a bijection onto $\mathbb N$, and
\[
n\longmapsto\varepsilon n
\]
is a bijection from $\mathbb N$ onto $[\varepsilon\infty]$.

The coordinate map of $R=S\circ U$ is
\[
r=s\circ\vartheta_U,
\]
and is therefore a bijection onto $\mathbb N$. Hence $R\in\mathcal I_{p-1+q}^\varepsilon$.

To prove associativity, let
\[
V=
\bigl(
v(k_1,\ldots,k_r)
\bigr)_{
k_1,\ldots,k_r\in[\varepsilon\infty]
}
\in\mathcal I_r^\varepsilon.
\]
The coordinate function of $(S\circ U)\circ V$ is
\[
\begin{aligned}
&
s\bigl(
i_1,\ldots,i_{p-1},
\varepsilon
u(
j_1,\ldots,j_{q-1},
\varepsilon v(k_1,\ldots,k_r)
)
\bigr).
\end{aligned}
\]
This is also the coordinate function of $S\circ(U\circ V)$. Therefore,
\[
(S\circ U)\circ V
=
S\circ(U\circ V).
\]

Finally, the coordinate function of $E^\varepsilon$ is $e(i)=\varepsilon i$.
Since $\varepsilon^2=1$, Definition~\ref{def:composition} gives
\[
S\circ E^\varepsilon=S
\qquad\text{and}\qquad
E^\varepsilon\circ S=S.
\]
Thus $E^\varepsilon$ is an identity.
\end{proof}

\begin{proposition}\label{P:involution-composition}
For every $\varepsilon\in\{+,-\}$ and $d\geq1$, the involution
is restricted to a bijection $(\cdot)^*:
\mathcal I_d^\varepsilon
\longrightarrow
\mathcal I_d^{-\varepsilon}$. Moreover, if $S\in\mathcal I_p^\varepsilon$ and $U\in\mathcal I_q^\varepsilon$,
then $(S\circ U)^*=S^*\circ U^*$. Consequently,
$(\cdot)^*:
(\mathcal I^\varepsilon,\circ)
\longrightarrow
(\mathcal I^{-\varepsilon},\circ)$ is an isomorphism of $\mathbb N$-graded monoids.
\end{proposition}

\begin{proof}
The involution preserves the support and injectivity of a uniform
element, preserves its degree, and reverses its sign. Therefore, it maps
$\mathcal I_d^\varepsilon$ bijectively to
$\mathcal I_d^{-\varepsilon}$.

Let
\[
S=
\bigl(s(i_1,\ldots,i_p)\bigr)_{
i_1,\ldots,i_p\in[\varepsilon\infty]}
\in\mathcal I_p^\varepsilon
\]
and
\[
U=
\bigl(u(j_1,\ldots,j_q)\bigr)_{
j_1,\ldots,j_q\in[\varepsilon\infty]}
\in\mathcal I_q^\varepsilon.
\]
For
\[
a_1,\ldots,a_{p-1},b_1,\ldots,b_q
\in[-\varepsilon\infty],
\]
we have
\begin{align*}
&(S\circ U)^*
 (a_1,\ldots,a_{p-1},b_1,\ldots,b_q)
\\
&\quad=
(S\circ U)
(-a_1,\ldots,-a_{p-1},-b_1,\ldots,-b_q)
\\
&\quad=
s\bigl(
-a_1,\ldots,-a_{p-1},
\varepsilon u(-b_1,\ldots,-b_q)
\bigr).
\end{align*}
On the other hand, since $S^*$ and $U^*$ have sign
$-\varepsilon$,
\begin{align*}
&(S^*\circ U^*)
(a_1,\ldots,a_{p-1},b_1,\ldots,b_q)
\\
&\quad=
s^*\bigl(
a_1,\ldots,a_{p-1},
-\varepsilon u^*(b_1,\ldots,b_q)
\bigr)
\\
&\quad=
s\bigl(
-a_1,\ldots,-a_{p-1},
\varepsilon u(-b_1,\ldots,-b_q)
\bigr).
\end{align*}
Therefore, $(S\circ U)^*=S^*\circ U^*$. Finally,
$(E^\varepsilon)^*=E^{-\varepsilon},$
and the involution preserves the reduced degree $\ell(S)=\delta(S)-1$.

Thus it is an isomorphism of $\mathbb N$-graded monoids.
\end{proof}

The restriction to bench-free uniform elements of the same sign is
structural. Benches have no final infinite coordinate into which
another indexing pattern can be substituted, while mixed
substitution generally produces variable-depth index patterns. More
general forms of composition are left for future work.

\begin{exmp}\label{Ex:represent}
Let $T=\left( 1,0,3,2,5,4,\ldots \right) \in \mathcal{I}_{1}^{+}$, that 
is, $t\left( i\right) =i+1$ for even $i$ and $t\left( i\right) =i-1$ for 
odd $i$, and let $S\in \mathcal{I}_{2}^{+}$ be given by
\begin{equation*}
s\left( i_{1},i_{2}\right) =\left( 2i_{2}+1\right) 2^{i_{1}}-1,\qquad 
i_{1},i_{2}\in \left[ +\infty \right] ,
\end{equation*}%
that is,
\begin{equation*}
S=\left( \left( 0,2,4,6,\ldots \right) ,\left( 1,5,9,13,\ldots \right)
,\left( 3,11,19,\ldots \right) ,\ldots \right) ,
\end{equation*}%
which belongs to $\mathcal{I}_{2}^{+}$, since every natural number $n$ 
admits a unique factorisation
$n+1=\left( 2i_{2}+1\right) 2^{i_{1}}$.

\begin{enumerate}
\item[$(a)$] The composition $R=S\circ T\in \mathcal{I}_{2}^{+}$, 
given by
$r\left( i_{1},i_{2}\right) =s\left( i_{1},t\left( i_{2}\right) \right) $, 
reorders each row of $S$:
\begin{equation*}
R=\left( \left( 2,0,6,4,\ldots \right) ,\left( 5,1,13,9,\ldots \right)
,\ldots \right) ,
\end{equation*}%
whereas $V=T\circ S\in \mathcal{I}_{2}^{+}$, given by
$v\left( i_{1},i_{2}\right)
=t\left( s\left( i_{1},i_{2}\right) \right) $, replaces each entry of
$S$ by its image under $T$:
\begin{equation*}
V=\left( \left( 1,3,5,7,\ldots \right) ,\left( 0,4,8,12,\ldots \right)
,\left( 2,10,18,\ldots \right) ,\ldots \right) .
\end{equation*}%
Both elements belong to $\mathcal{I}_{2}^{+}$ and therefore have order
type $\omega ^{2}$. Indeed, $p-1+q=2$ for
$\left( p,q\right) =\left( 2,1\right) $ and
$\left( 1,2\right) $, respectively. Nevertheless, $R\neq V$, so the
composition is not commutative.

\item[$(b)$] Composing two elements of degree $2$ increases the degree. 
Indeed, $W=S\circ S\in \mathcal{I}_{3}^{+}$ is given by
$w\left(
i_{1},j_{1},j_{2}\right) =s\left( i_{1},s\left( j_{1},j_{2}\right)
\right) $. For instance,
$w\left( 0,1,0\right) =s\left( 0,s\left(
1,0\right) \right) =s\left( 0,1\right) =2$. Moreover, $W$ has order
type $\omega ^{3}$, in accordance with
$p-1+q=2-1+2=3$.
\end{enumerate}
\end{exmp}

Example~\ref{Ex:represent} exhibits both the noncommutativity of the
composition and its effect on the degree. The next result gives the
general order-type formula.

\begin{corollary}\label{order}
Let $S\in \mathcal{I}_{p}^{\varepsilon }$ and
$U\in \mathcal{I}_{q}^{\varepsilon }$. Then
\begin{equation*}
\ot(S\circ U)=
\begin{cases}
\omega^{p-1+q}, & \varepsilon=+,\\[1mm]
\left(\omega^{p-1+q}\right)^{\ast}, & \varepsilon=-.
\end{cases}
\end{equation*}
\end{corollary}

\begin{proof}
By Proposition~\ref{P:composition-monoid},
\begin{equation*}
S\circ U\in\mathcal{I}_{p-1+q}^{\varepsilon }.
\end{equation*}
The result now follows from Theorem~\ref{order-type}.
\end{proof}

In degree $1$, an element of $\mathcal{I}_{1}^{\varepsilon }$ is
simply a bijective enumeration of $\mathbb N$, indexed by
$[\varepsilon\infty]$. More precisely, every bijection
$\pi:\mathbb N\rightarrow\mathbb N$ determines an element
\begin{equation*}
\Phi_{\varepsilon}(\pi)
:=
\left(\pi(\varepsilon i)\right)_{i\in[\varepsilon\infty]}
\in\mathcal{I}_{1}^{\varepsilon }.
\end{equation*}
Conversely, every element of $\mathcal{I}_{1}^{\varepsilon }$ is
obtained uniquely in this way. Moreover, for
$\pi,\rho\in\operatorname{Sym}(\mathbb N)$,
\begin{equation*}
\Phi_{\varepsilon}(\pi)\circ\Phi_{\varepsilon}(\rho)
=
\Phi_{\varepsilon}(\pi\circ\rho).
\end{equation*}

The above identity shows that the degree-one component reproduces
ordinary permutation composition. We establish this correspondence as a
group isomorphism, which will also be used to identify the units of
the $\mathbb{N}$-graded monoid $(\mathcal I^\varepsilon,\circ)$.

\begin{theorem}\label{T:isogroups}
For each $\varepsilon\in\{-,+\}$, the set
$\mathcal I_1^\varepsilon$ is a group under composition. Moreover, the
map
\begin{equation*}
\Phi_\varepsilon:
\mathrm{Sym}(\mathbb N)
\longrightarrow
\mathcal I_1^\varepsilon
\end{equation*}
defined by
\begin{equation*}
\Phi_\varepsilon(\pi)
:=
\left(
\pi(\varepsilon i)
\right)_{i\in[\varepsilon\infty]}
\end{equation*}
is a group isomorphism. Consequently,
\begin{equation*}
\left(\mathcal I_1^+,\circ\right)
\cong
\left(\mathcal I_1^-,\circ\right)
\cong
\mathrm{Sym}(\mathbb N).
\end{equation*}
\end{theorem}

\begin{proof}
By Proposition~\ref{P:composition-monoid},
$\mathcal I_1^\varepsilon$ is closed under composition, the
composition is associative, and
\begin{equation*}
E^\varepsilon
=
\left(
\varepsilon i
\right)_{i\in[\varepsilon\infty]}
\end{equation*}
is its identity element.

The map $\Phi_\varepsilon$ is bijective. Indeed, if
\begin{equation*}
S=
\left(
s(i)
\right)_{i\in[\varepsilon\infty]}
\in\mathcal I_1^\varepsilon,
\end{equation*}
then the map
\begin{equation*}
\pi_S:\mathbb N\longrightarrow\mathbb N,
\qquad
\pi_S(n):=s(\varepsilon n),
\end{equation*}
is a bijection and satisfies
$\Phi_\varepsilon(\pi_S)=S$.

Finally, let $\pi,\rho\in\mathrm{Sym}(\mathbb N)$. For every
$j\in[\varepsilon\infty]$, Definition~\ref{def:composition} gives
\begin{equation*}
\begin{aligned}
\left(
\Phi_\varepsilon(\pi)
\circ
\Phi_\varepsilon(\rho)
\right)(j)
&=
\pi\left(
\varepsilon^2\rho(\varepsilon j)
\right)\\
&=
\pi\left(
\rho(\varepsilon j)
\right)\\
&=
\Phi_\varepsilon(\pi\circ\rho)(j).
\end{aligned}
\end{equation*}
Thus $\Phi_\varepsilon$ is a group isomorphism.
\end{proof}

\begin{corollary}\label{P:invertible}
For each $\varepsilon\in\{-,+\}$, the group of units of the monoid
$\left(\mathcal I^\varepsilon,\circ\right)$ is precisely
$\mathcal I_1^\varepsilon$.
\end{corollary}

\begin{proof}
Let $S\in\mathcal I_p^\varepsilon$ be invertible, with inverse
$U\in\mathcal I_q^\varepsilon$. Then
\begin{equation*}
S\circ U=E^\varepsilon\in\mathcal I_1^\varepsilon.
\end{equation*}
By Proposition~\ref{P:composition-monoid},
\begin{equation*}
S\circ U\in\mathcal I_{p-1+q}^\varepsilon.
\end{equation*}
Since the degree of a uniform element is unique, it follows that $p-1+q=1$.
As $p,q\geq1$, we obtain $p=q=1$. Conversely, every element of
$\mathcal I_1^\varepsilon$ is invertible by
Theorem~\ref{T:isogroups}.
\end{proof}
\section{Conclusions and further research}
\label{S:Conclusions}

\subsection*{Summary of contributions}

This paper develops a finite-address framework for layered
decomposable shuffles, laying the foundations for extensions to
broader classes of shuffles and for the further study of their
algebraic structure and classification. The present work focuses on
representation, induced order types, duality, scatteredness, and
composition. Uniform elements encode order types $\omega^d$ and
$(\omega^d)^*$, whereas mixed elements concatenate components of
different degrees and signs. The address map of
Definition~\ref{D:address} reduces comparison in the induced order on
$\mathbb N$ to lexicographic comparison of finite tuples.

Theorem~\ref{order-type} determines the order types represented by
uniform and mixed elements. The representability results also show
that every order induced by an element of $\mathcal S_\infty$ is
scattered. Moreover, the dense order type $\eta$ cannot be
induced by an element of $\mathcal S_\infty$, as shown in
Theorem~\ref{T:scattered} and Corollary~\ref{C:eta}.

The involution introduced in Section~\ref{S:Involution} reverses the sign
of each uniform component and, for mixed elements, the order of the outer index set.
Consequently, every represented order has a represented order dual,
as shown in Theorem~\ref{T:inv-order}.

Finally, Section~\ref{S:Composition} introduces a hierarchical
composition on bench-free uniform elements of fixed sign. For each
$\varepsilon\in\{+,-\}$, the resulting structure
$(\mathcal I^\varepsilon,\circ)$ is a monoid graded by the reduced
degree
\[
\ell(S):=\delta(S)-1,
\qquad
\ell(S\circ U)=\ell(S)+\ell(U).
\]
The order type of a composition is determined by
Corollary~\ref{order}. In degree $1$, the composition coincides with
ordinary permutation composition, giving
\[
(\mathcal I_1^\varepsilon,\circ)
\cong
\operatorname{Sym}(\mathbb N),
\]
and the degree-$1$ elements are precisely the units of
$\mathcal I^\varepsilon$.

\subsection*{Further directions}

\medskip
\noindent\textbf{1. Classification and normal forms.}
A natural next step is to classify, up to order isomorphism, the
orders induced by decomposable shuffles. This will require
appropriate normal forms and invariants capable of accounting for
regrouping, absorption, and the non-uniqueness of consecutive
snake--ladder decompositions.

\medskip
\noindent\textbf{2. Algebraic structure and mixed actions.}
For each sign $\varepsilon$, the ideals, generators, and presentations
of the graded monoid $(\mathcal I^\varepsilon,\circ)$ remain to be
studied. The substitution formula extends in the first argument to
injective uniform elements of arbitrary support and consequently
induces a natural right action of $\mathcal I^\varepsilon$ on suitable
bench-free mixed elements. The orbits and stabilisers of this action
provide further problems for investigation.

\medskip
\noindent\textbf{3. Broader representations and mixed composition.}
Allowing mixed components to occur recursively, and permitting more
general countable scattered index orders, may extend the
finite-address formalism beyond the layered subclass. Such an
extension could also provide a natural setting for the composition of
two bench-free mixed elements and for bilateral representations, such
as those associated with the order type
$\zeta=\omega^*+\omega$. The resulting algebraic structure would
likely be governed by substitution of address patterns rather than by
a single numerical degree.

\medskip
\noindent\textbf{4. Dynamical origins of higher-degree shuffles.}
The Sharkovski\u{\i} order~\eqref{E:Shar_order} arises from the
coexistence of periodic orbits of continuous interval
maps~\cite{Shar}. It remains to determine whether higher-degree or
mixed shuffles arise naturally from other dynamical forcing
structures, and whether their degrees and signs encode meaningful
dynamical information.

\section*{Funding}
This work is funded by national funds through the FCT – Fundação para a Ciência e a Tecnologia, I.P., under the scope of the projects\\ \url{https://doi.org/10.54499/UID/04674/2025} (all authors) and \\\url{https://doi.org/10.54499/UID/04561/2025} (second author).

\section*{Acknowledgements}
The three authors acknowledge the support of the research centre  CIMA - Centro de Investigação em Matemática e Aplicações. The second author also acknowledges the support of the research center CEMS.UL – Centro de Estudos Matemáticos de Ciências, ULisboa.

\bibliographystyle{plain}
\bibliography{References}

\end{document}